\documentclass{amsart}
\usepackage{amssymb,amsmath,latexsym,dsfont,tikz,caption,subcaption,thm-restate,hyperref,cleveref,enumitem}
\usepackage{graphicx,verbatim,ifpdf,mdwlist}
\usepackage{xcolor}
\usepackage{tcolorbox}
\tcbuselibrary{breakable}
\usetikzlibrary{arrows,calc,shapes,decorations.pathreplacing,positioning}
\usepackage{color}
\usepackage{mathabx}
\ifpdf
\usepackage{hyperref}
\else
\usepackage[hypertex]{hyperref}
\fi

\newtheorem{theorem}{Theorem}[section]

\newtheorem{lemma}[theorem]{Lemma}

\newtheorem{proposition}[theorem]{Proposition}

\newtheorem{conjecture}[theorem]{Conjecture}

\newtheorem{definition}[theorem]{Definition}

\newtheorem{exa}[theorem]{Example}

\newtheorem{rem}[theorem]{Remark}

\title[Relatively acceptable notation]{Relatively acceptable notation}
\author[N.~Bazhenov]{Nikolay Bazhenov}
\address{Sobolev Institute of Mathematics, pr. Akad.
Koptyuga 4, Novosibirsk, 630090 Russia
}
\email{bazhenov@math.nsc.ru}

\author[D.~Kalociński]{Dariusz Kalociński}
\address{Institute of Computer Science, Polish Academy of Sciences, ul. Jana Kazimierza 5, 01-248 Warsaw, Poland}\email{dariusz.kalocinski@ipipan.waw.pl}

\thanks{Bazhenov is supported by the Mathematical Center in Akademgorodok under the agreement No. 075-15-2022-281 from 05.04.2022 with the Ministry of Science and Higher Education of the Russian Federation. Kalociński was supported by the National Science Centre grant no. 2018/31/B/HS1/04018.}



\begin{document}

\maketitle

\begin{abstract}
Shapiro's notations for natural numbers, and the associated desideratum of acceptability---the property of a notation that all recursive functions are computable in it---is well-known in philosophy of computing. Computable structure theory, however, although capable of fully reconstructing Shapiro's approach, seems to be off philosophers' radar. Based on the case study of natural numbers with standard order, we make initial steps to reconcile these two perspectives. First, we lay the elementary conceptual groundwork for the reconstruction of Shapiro's approach in terms of computable structures and show, on a few examples, how results pertinent to the former can inform our understanding of the latter. Secondly, we prove a new result, inspired by Shapiro's notion of acceptability, but also relevant for computable structure theory. The result explores the relationship between the classical notion of degree spectrum of a computable function on the structure in question---specifically, having all c.e.\ degrees as a spectrum---and our ability to compute the (image of the) successor from the (image of the) function in any computable copy of the structure. The latter property may be otherwise seen as relativized acceptability of every notation for the structure.
\end{abstract}

\section{Introduction}\label{sec:introduction}
Shapiro made a point that computations are performed on syntactic objects, such as strings of symbols, rather than on numbers themselves \cite{shapiro_acceptable_1982}. Some models of computation are directly based on this premise \cite{turing_computable_1936,markov_theory_1951} but some are not \cite{shepherdson_computability_1963,church_unsolvable_1936}. When showing equivalence between these models, one uses some form of notation for natural numbers, e.g., the unary notation. However, as Shapiro observed, not every notation is appropriate for showing the desired equivalence. This made him to ask when a notation is appropriate in the above sense, or, in his words, when it should be deemed acceptable. He showed that a notation is acceptable if and only if the successor is computable in it.

Shapiro's notations became influential in philosophy of computing (see, e.g., \cite{michael_rescorla_churchs_2007,copeland_deviant_2010,quinon_taxonomy_2018,wroclawski_michal_representations_2019,shapiro_computability_2022}). Meanwhile, we have seen rapid development of computable structure theory which explores the relationship between computability of countable objects and algebraic structures \cite{ash_computable_2000,montalban_computable_2021}. Interestingly, Shapiro's framework can be fully rephrased in terms of computable structures under the following slogan: a notation for a given computable structure corresponds to a computable isomorphic copy of the structure. This allegedly simple fact, and its consequences, has been consistently overlooked by philosophers, probably due to differing terminologies, emphasis on apparently different notions, as well as intricate technicalities of the theory itself. Anyhow, in this paper we would like to partly reconcile these two perspectives. 
We will do it in Section \ref{sec:notations}, which is intended as a gradual introduction of the necessary concepts, assuming basics of computability theory \cite{cutland}, and in Appendix \ref{sec:A}, which recasts some results of Shapiro. Our exposition, although restricted to the structure $(\omega,<)$, i.e., natural numbers with the standard order, generalizes to other computable structures. This is the first contribution of the paper which should appeal to formally inclined philosophers.

The second contribution of the paper explores the connection between two properties, namely \ref{p1} the degree spectrum of all and (only) computably enumerable (c.e.) degrees, and \ref{p2} the recoverability of the successor. 

The notion of \emph{degree spectrum} is a classical one \cite{richter_degrees_1977,richter_degrees_1981,harizanov_degree_1987}. The degree spectrum of computable $R$ on $(\omega,<)$ is the set of all Turing degrees of the images of $R$ in all computable isomorphic copies of $(\omega,<)$; the spectrum of all c.e. degrees is the degree spectrum of the successor on $(\omega,<)$ \cite{chubb_degree_2009}. On the other hand, the \emph{recoverability of the successor} is inspired by Shapiro's desideratum of acceptability of notation \cite{shapiro_acceptable_1982}. Recall that a notation is acceptable if the successor function is computable in it. We look at a relativized version of acceptability of a computable isomorphic copy $\mathcal A$ of $(\omega,<)$, namely the computability of $Succ_\mathcal{A}$ (the image of the successor in $\mathcal A$) relative to $R_\mathcal{A}$ (the image of $R$ in $\mathcal A$). We say that the successor is recoverable from $R$ on $(\omega,<)$ if $Succ_\mathcal{A}$ is computable from $R_\mathcal{A}$ in any computable isomorphic copy $\mathcal A$ of $(\omega,<)$. 

It is known that the recoverability of the successor from $R$ on $(\omega,<)$ fixes $R$'s spectrum to that of all c.e. degrees (see, e.g., \cite{bazhenov_intrinsic_2022}).
On the other hand, for any known example of a computable relation $R$ having this degree spectrum, the successor is recoverable from $R$. We ask whether this holds in general, namely whether for any computable relation $R$, if $R$'s degree spectrum on $(\omega,<)$ is equal to all c.e. degrees then the successor is recoverable from $R$ on $(\omega,<)$. Following \cite{bazhenov_intrinsic_2022}, we attack this problem for a restricted but nevertheless inclusive class of computable relations, namely the graphs of unary total computable functions. We expand on techniques developed there, in particular by reusing the concept of block functions. Our result---which answers the aforementioned question in the affirmative---is proved for computable block functions satisfying certain intuitive effectiveness condition (Theorem~\ref{theorem:non-c.e.}). The case of computable block functions that do not satisfy it is left open.

\section{Notations and computable structures}\label{sec:notations}

Originally, Shapiro considered notations for natural numbers (henceforth, $\omega$) with no additional structure. A \emph{notation} for $\omega$ is any bijection $\sigma: S \to \omega$, where $S$ is a fixed countably infinite set of strings over a finite alphabet. The idea is that $\alpha \in S$ is a numeral denoting $\sigma(\alpha)$ and we think of computations as being performed only on numerals. Computability of an $n$-ary relation $R$ \emph{on natural numbers} in $\sigma$ is equated, by definition, with the computability of the $\sigma$-preimage of $R$, $\{(\sigma^{-1}(a_1), \sigma^{-1}(a_2), \ldots, \sigma^{-1}(a_n)): (a_1,a_2,\ldots,a_n) \in R\}$, which is a relation \emph{on numerals}. 

Observe that, without loss of generality, we can effectively represent any infinite computable $S \subseteq \Sigma^*$ as $\omega$. Such identification is commonplace is computability theory. Essentially, it boils down to two facts. One is that, for any finite $\Sigma$, there exists a computable bijection $c: \Sigma^* \to \omega$. Since $S$ is computable, it is coded, as we say, by the computable set $c(S)$ (the image of $S$ under $c$). Second is that for every computable infinite set $A \subseteq \omega$ there exists a strictly increasing computable enumeration of $A$, i.e., a total computable function $g$ such that the image of $g$ is $A$ and $g(n) < g(n+1)$, for all $n \in \omega$. Given these two facts, the effective representation of $S$ as $\omega$ is established as follows. Take a strictly increasing computable enumeration of $c(S)$, say $n_0 < n_1 < n_2 < \ldots$, and substitute each $n_i$ by $c^{-1}(n_i)$ (note that $c^{-1}$ is a computable function). This yields a computable sequence $s_0, s_1, s_2, \ldots$ which consists of all (and only) elements of $S$, without repetitions. We take such identification for granted and thus we can assume that a notation for $\omega$ is any bijection from $\omega$ to $\omega$.

Shapiro considered only notations for plain $\omega$ (in Appendix \ref{sec:A}, we reproduce some of his results in the setting of computable structure theory). However, one can easily extend his notion to cover additional structure. We shall focus on the additional structure in the form of the simplest ordering possible---the standard order on natural numbers, which we denote by $<$. Hence, the structure under investigation is $(\omega,<)$. By a notation for $(\omega,<)$ we shall mean any notation $\sigma$ for $\omega$ in which $<$ is computable.

Now, let us turn our attention to computable structures.
A countably infinite relational structure $(A,R_1,\dots,R_n)$ is said to be \emph{computable} if its universe $A$ and each relation $R_i$ are computable. Without loss of generality, we can assume that $A = \omega$, the reason being similar to the one already discussed for $S$. Clearly, $(\omega,<)$ is a computable structure. 

To pinpoint the connection between notations and computable structures, consider the following definition.
\begin{definition}\label{df_copy}
$(\omega, \prec)$ is a computable copy of $(\omega,<)$ if $\prec$ is a computable ordering on $\omega$ and structures $(\omega, <)$ and $(\omega, \prec)$ are isomorphic.
\end{definition}

Now, we can make the following observations. Let $\sigma: \omega \to \omega$ be a notation for $(\omega,<)$ and let $\prec$ be the $\sigma$-preimage of $<$, i.e., $\prec := \{(\sigma^{-1}(x),\sigma^{-1}(y)): x < y\}$. By the definition of notation, $\prec$ is computable. Moreover, by the definition of $\prec$, the structures $(\omega,<)$ and $(\omega,\prec)$ are isomorphic. Therefore, by Definition \ref{df_copy}, $(\omega,\prec)$ is a computable copy of $(\omega,<)$.

The next observation goes in the other direction. Let $(\omega,\prec)$ be a computable copy of $(\omega,<)$. Let $h:(\omega,<) \cong (\omega,\prec)$ be an isomorphism between the two structures and let $\sigma = h^{-1}$. Obviously, $\sigma$ is a notation for $\omega$. Also, the $\sigma$-preimage of $<$ is equal to $\prec$ and, by Definition \ref{df_copy}, $\prec$ is computable. Therefore, $\sigma$ is a notation for $(\omega,<)$.

Based on the above two observations, we can posit that a notation for $(\omega,<)$ is any isomorphism that maps a computable structure $(\omega,\prec)$ to $(\omega,<)$. 
 Therefore, instead of notations for $(\omega,<)$ we may equivalently speak about computable copies of $(\omega,<)$. This generalizes straightforwardly to arbitrary computable structures.
 
Let us recall one of the concepts introduced by Shapiro, namely acceptability of notation, and see what does it mean in terms of computable structures. A notation for $\omega$ is said to be \emph{acceptable} if the successor function (henceforth, $Succ$) is computable in it (this implies that all recursive functions are). Acceptability of notation for plain $\omega$ can be extended, in an obvious way, to notations for $\omega$ with additional structure, in particular $(\omega,<)$. 

 What does the desideratum of notation's acceptability mean from the perspective of computable structures? Before we answer this, consider the following convention which is commonplace when referring to isomorphic copies of a given structure.
\begin{definition}
Let $R$ be a relation on $(\omega,<)$, i.e., $R \subseteq \omega^k$, for some $k \in \omega$, and let $\mathcal A$ be a computable copy of $(\omega,<)$. If $\varphi$ is an isomorphism from $(\omega,<)$ to $\mathcal A$, we write $R_\mathcal{A}$ for the image of $R$ under $\varphi$. 
\end{definition}
Now, in terms of computable structures, the desideratum of acceptability of a computable copy $\mathcal A = (\omega,<_{\mathcal{A}})$---which uniquely identifies a notation for $(\omega,<)$---says that $Succ_\mathcal{A}$ should be computable.


Shapiro showed, among others, that not every notation for $\omega$ is acceptable. The following question arises immediately: is $Succ_\mathcal{A}$ computable, if $\mathcal A$ is any computable copy $\mathcal A $ of $(\omega,<)$? In other words, is every computable copy of $(\omega,<)$ acceptable? In an influential philosophical paper, Benaceraff hinted in passing at the affirmative (cf. p. 276 in \cite{benacerraf_what_1965} and, also, \cite{benacerraf_recantation_1996}). Essentially, his claim was that, in any given notation, the computability of the successor is equivalent to the computability of the ordering (we take the liberty to identify Benaceraff's intuitive concept of notation with the formal one).  

As we will see, Benaceraff's claim is false in view of Proposition \ref{theorem:DgSp_Succ} below. The proposition uses an important notion of degree spectrum, originating from Richter \cite{richter_degrees_1977,richter_degrees_1981} and Harizanov \cite{harizanov_degree_1987} (see, also, \cite{hirschfeldt_degree_2000,fokina_comp_model_20l4}). The notion of degree spectrum---here defined for $(\omega,<)$ but, in general, applicable to any computable structure---is based on the concept of Turing degrees. 

As a brief reminder, Turing degrees are equivalence classes of the relation $\equiv_T$ on subsets of $\omega$, defined by $A \equiv_T B \Leftrightarrow A \leq_T B \wedge B \leq_T A$, where $X \leq_T Y$ means, roughly, that one can compute $X$ by an algorithm which, from time to time, may request the answer to any question of the form \emph{Does $n$ belong to $Y$?}. For more information, see, e.g., \cite{cutland,rogers_theory_1967,soare_recursively_1987}.

\begin{definition}\label{dgsp1}
The degree spectrum of a computable relation $R$ on $(\omega,<)$, in symbols $DgSp(R)$, is the set of Turing degrees of $R_\mathcal{A}$ over all computable copies $\mathcal A$ of $(\omega,<)$. 
\end{definition}
As we can see, the notion of degree spectrum encompasses all possible complexities of a relation (formalized as Turing degrees) across all notations for the underlying structure. 

If Benaceraff's claim was true, $Succ_\mathcal{A}$ would be computable for each computable copy $\mathcal A$ of $(\omega,<)$ and, therefore, we would have $DgSp(Succ) = \{\mathbf 0\}$, where $\mathbf{0} = \{A : A \equiv_T \emptyset\}$ is the degree of decidable sets. This, however, would contradict the proposition below. (Recall that a Turing degree is computably enumerable (c.e.) if it contains a computably enumerable set, i.e., a set which, if nonempty, can be enumerated by an algorithm. It is folklore that there are c.e. sets which are not computable and, therefore, that there are c.e. degrees which are different than $\mathbf 0$.)
\begin{proposition}[see, e.g., Example 1.3 in \cite{chubb_degree_2009}]\label{theorem:DgSp_Succ}
The spectrum of successor on $(\omega,<)$ consist of all (and only) c.e. degrees.
\end{proposition}
Therefore, not every computable copy $\mathcal A$ of $(\omega,<)$ is acceptable.

The degree spectrum of the successor on $(\omega,<)$, i.e., the set consisting of all c.e. degrees, is just one example of degree spectrum on $(\omega,<)$, but other examples exist and it is still not known what are all possibilities. One kind of degree spectrum is the trivial one, i.e., $\{\mathbf 0\}$. Relations having this spectrum are called intrinsically computable and were characterized by Moses \cite{moses_relations_1986}. 
There are also other kinds of spectra, including the set of all $\Delta_2$ degrees \cite{downey_degree_2009,wright_degrees_2018,harrison-trainor_degree_2018} and other, discovered quite recently \cite{bazhenov_intrinsic_2022}.

Anyway, the degree spectrum of the successor is, in a sense, special. Essentially, $DgSp(Succ) \subseteq DgSp(R)$ for every computable $R$ which is not intrinsically computable \cite{wright_degrees_2018}. 
To show that $DgSp(R) = DgSp(Succ)$, one typically uses a technique known as \emph{recovering the successor} \cite{bazhenov_intrinsic_2022}. The idea is to prove that $R_\mathcal{A} \geq_T Succ_\mathcal{A}$ holds for every computable copy $\mathcal A$ of $(\omega,<)$; this is sufficient because $Succ_\mathcal{A} \geq_T R_\mathcal{A}$ always, i.e., for every computable copy $\mathcal A$ of $(\omega,<)$. In a sense, such a proof shows how the successor (inside $\mathcal A$) can be recovered from the relation $R$ (again, inside $\mathcal A$). If this can be done, we say that the successor is recoverable from $R$ on $(\omega,<)$.

In the second part of the paper, we are interested whether one can go in the other direction. Specifically, we ask whether $DgSp(R) = DgSp(Succ)$ implies that the successor is recoverable from $R$ on $(\omega,<)$. In other words, given the following properties of a computable relation $R$:
\begin{enumerate}[label=(\Roman*)]
    \item the degree spectrum of $R$ on $(\omega,<)$ is equal to all c.e. degrees, and \label{p1}
    \item for every computable copy $\mathcal A$ of $(\omega,<)$, $R_\mathcal{A} \geq_T Succ_\mathcal{A}$,\label{p2}
\end{enumerate}
we ask whether \ref{p1} implies \ref{p2} (the other direction follows from one of the paragraphs above). Note that \ref{p2} can be seen as a relativized variant of acceptability of notation. We may posit that a computable copy (notation) $\mathcal A = (\omega,<_\mathcal{A})$ is acceptable relative to $R$ if $R_\mathcal{A}$ computes $Succ_\mathcal{A}$. Thus, the property in question encompasses computable relations such that every computable copy $(\omega,<_\mathcal{A})$ is acceptable relative to them.

 Although we are not able to prove this implication in full generality, we show that it holds for a wide subclass of all total computable functions (seen as computable binary relations). For certain types of functions, the implication already follows from earlier results which will be discussed in the next section. Our extension concerns the class of block functions, isolated in this context by Bazhenov et al. \cite{bazhenov_intrinsic_2022}, and satisfying an additional effectiveness condition.



\section{Spectrum of c.e. degrees and recovering the successor}\label{sec:result}
Following Bazhenov et al. \cite{bazhenov_intrinsic_2022}, we restrict our attention to computable binary relations $R$ of general interest---graphs of unary total computable functions. The results of Moses \cite{moses_relations_1986} imply that the degree spectrum of such an $f$ is trivial if and only if $f$ is almost constant or almost identity (see, also, the associated version of \cite{bazhenov_intrinsic_2022}). Clearly, the equivalence $\text{\ref{p1}}\Leftrightarrow\text{\ref{p2}}$ holds for such functions as they do not satisfy neither \ref{p1} nor \ref{p2}. Bazhenov et al. \cite{bazhenov_intrinsic_2022} isolated the class of quasi-block functions (which include almost constant and almost identity functions) and showed that all computable functions $f$ outside this class satisfy \ref{p1} and \ref{p2} (Theorem 18 in \cite{bazhenov_intrinsic_2022}). Recall that $f$ is a quasi-block function if there are arbitrarily long initial segments of $(\omega,<)$ which are closed under $f$. Further information about quasi-block functions, pertinent also to the problem at hand, can be found in \cite{bazhenov_intrinsic_2022}. 

Among quasi-block functions there is a narrower class of block functions.
\begin{definition}\label{def:block_function}
    Let $f\colon\omega\to\omega$ be a total function. An interval $I$ of $(\omega,<)$ is \emph{$f$-closed} if for all $x\in I$, $f(x)\in I$ and $f^{-1}(x)\subseteq I$. For a finite non-empty interval $I\subset\omega$, the structure $(I, <, f\upharpoonright I)$ is an \emph{$f$-block} if it has the following properties:\begin{itemize}
        \item $I$ is an $f$-closed interval and it cannot be written as a disjoint union of several $f$-closed intervals;
        \item $\{ x\in\omega\,\colon x < I\}$ is $f$-closed.
    \end{itemize}
    The function $f$ is a \emph{block function} if for every $a \in \omega$, there is an $f$-block containing $a$. If $(I, <, f\upharpoonright I)$ is an $f$-block, we refer to its isomorphism type as an \emph{$f$-type} (or a type).
\end{definition}

For a set $X$, by $\mathrm{card}(X)$ we denote the cardinality of $X$. 
Given a computable block function $f$, let us define the cardinality-of-the-preimage function $cp_f$ by
    $cp_f(x) = \mathrm{card}(f^{-1}(x))$, for all $x \in \omega$. We will focus on computable block functions $f$ satisfying the following effectiveness condition:
\begin{equation}\label{effectiveness_condition}
    \text{$cp_{f}$ is computable.}\tag{$\star$}
\end{equation}
For example, each computable block function $f$ with only finitely many $f$-types satisfies (\ref{effectiveness_condition}). As a more sophisticated example, consider any injective computable block function (such a function is composed of finite cycles).

Before we formulate the main theorem, let us prove the following lemma which will be used later.
\begin{lemma}\label{lemma:successor-recovery}
Let $f$ be a unary total computable function. If there exists an infinite computable sequence $(\mathcal B_j)_{j\in\omega}$ of growing initial segments of $(\omega,<,f)$ such that each $\mathcal B_j$ embeds in $(\omega,<,f)$ precisely once, then the successor is recoverable from $f$ on $(\omega,<)$.
\end{lemma}
\begin{proof}
To compute $Succ_{\mathcal A}(x)$ for a given computable copy $\mathcal A = (\omega,<_\mathcal{A})$ of $(\omega,<)$, look for bigger and bigger substructures $\mathcal A_j$ of $(\omega,<_\mathcal{A}, f_\mathcal{A})$ that look precisely as $\mathcal B_j$. We do this by enumerating $\omega$ one by one, arranging enumerated elements according to $<_\mathcal{A}$ and asking $f_\mathcal{A}$ for the values of the enumerated elements. Once $\mathcal A_j$ isomorphic to $\mathcal B_{j}$ is discovered, we are sure that $\mathcal A_j$ is an initial segment of $(\omega,<_\mathcal{A}, f_\mathcal{A})$. We continue in this manner and wait for $x$ and at least one $y >_\mathcal{A} x$ to enter some $\mathcal A_j$. At this point we know the position of $x$ in $<_\mathcal{A}$ and $Succ_\mathcal{A}(x)$ can be read out from $\mathcal A_j$.
\end{proof}

We prove \ref{p1}$\Rightarrow$\ref{p2} by contraposition. Notice that we can additionally assume that we work with computable block functions which are not almost identities because, as discussed earlier, such functions have the trivial degree spectrum. 

\begin{theorem}\label{theorem:non-c.e.}
    Let $f$ be a computable block function such that it is not almost identity and it satisfies the effectiveness condition (\ref{effectiveness_condition}). Then \ref{p1} implies \ref{p2}, i.e.: if the successor is not recoverable from $f$ on $(\omega,<)$,
    then the degree spectrum of $f$ on $(\omega,<)$ contains a non-c.e.\ degree.
\end{theorem}

%
%
        

\begin{proof}

Our construction builds a computable isomorphic copy $\mathcal{A} = (\omega, <_A)$ of the ordering $(\omega,<)$. For brevity, the graph  of the function $f_{\mathcal{A}}$ (i.e., the isomorphic image of our function $f$ inside $\mathcal{A}$) is denoted by $\Gamma_A$.

As usual, we fix an effective enumeration $(\Phi_i,\Psi_i, W_i)_{i\in\omega}$ containing all triples such that $\Phi_i$ and $\Psi_i$ are Turing functionals, and $W_i$ is a c.e.\ set. 
By Lemma~\ref{lemma:successor-recovery}, it is sufficient to satisfy the following series of requirements:
\begin{itemize}
    \item[$\mathcal{P}_i$:] If $W_i = \Phi_i^{\Gamma_A}$ and $\Gamma_A = \Psi_i^{W_i}$, then there exists a computable sequence $(\mathcal{B}_m)_{m\in\omega}$ of growing initial segments of $(\omega,<,f)$ such that each $\mathcal{B}_m$ isomorphically embeds into $(\omega,<,f)$ only once.
    
\end{itemize}
Indeed, since the successor is not recoverable from $f$, Lemma~\ref{lemma:successor-recovery} guarantees that the degree $\deg_T(\Gamma_A)$ is not c.e. Our strategy for satisfying $\mathcal{P}_i$ incorporates the classical construction of a properly d.c.e.\ Turing degree by Cooper~\cite{cooper_degrees_1971} adapted to the setting of block functions \cite{bazhenov_intrinsic_2022}. Its eventual success is secured by a series of threats which attempt to build $(\mathcal B_m)$. This will be a degenerate infinite injury construction, described using the framework of trees of strategies (see, e.g., \cite{lempp_priority_2012}).



For the sake of simplicity, first we give a construction for the case when each block is isomorphic to a cycle of the following form: 
\begin{itemize}
    \item the domain of a cycle contains numbers $0,1,\dots,k$ for some $k\geq 0$, 
    
    \item the ordering of the domain is standard, and 
    
    \item $f(i) = {i+1}$ for $i<k$, and $f(k) = 0$.
\end{itemize}
Note that in this case, we have $cp_{f}(x) = 1$ for all $x\in\omega$. 

The modifications needed for the general case of the theorem will be discussed in the end of the proof.


\textsc{Strategy for $\mathcal{P}_i$.} The strategy attempts to build a computable sequence of initial segments $(\mathcal{B}_m)_{m\in\omega}$. Suppose that our strategy starts working at a stage $s_0$. Then the strategy proceeds as follows.

(1)\ Set $m = 0$.
    
(2)\ Choose two adjacent blocks $\tilde{\mathcal{C}}_m < \tilde{\mathcal{D}}_m$ in $(\omega,<,f)$ such that we have not copied them into the structure $\mathcal{A}_{s_m}$ yet (where $s_m$ is the current stage), and $\tilde{\mathcal{D}}_m$ contains at least two elements. Such a choice is possible, since the function $f$ is not almost identity (hence, it has infinitely many cycles of size $\geq 2$).

We extend $\mathcal{A}_{s_m}$ to a finite structure $\mathcal{A}^1_m$ by copying all missing elements up to the end of the block $\tilde{\mathcal{D}}_m$ (these elements are appended to the end of $\mathcal{A}_{s_m}$). Then the structure $\mathcal{A}^1_m$ can be decomposed as follows:
\[
    \mathcal{A}^1_m = \mathcal{A}^{1,init}_m + \mathcal{C}_m + \mathcal{D}_m,
\]
where $\mathcal{C}_m \cong \tilde{\mathcal{C}}_m$ and $\mathcal{D}_m \cong \tilde{\mathcal{D}}_m$.

Let $x_m$ be the rightmost element of $\mathcal{C}_m$, and let $y_m$ be the leftmost element of $\mathcal{D}_m$. It is clear that at the moment, $\langle x_m,y_m\rangle \not\in \Gamma_A$. As usual, we restrict ``$\langle x_m,y_m\rangle \not\in \Gamma_A$" by forbidding lower priority actions to add new elements between the elements of $\mathcal{A}^1_m$.

(3)\ Wait for a stage $s'>s_m$ witnessing the following computations: for some $t'\leq s'$, we have (at the stage $s'$)
\[
    W_{i}\upharpoonright t' = \Phi_i^{\Gamma_A}\upharpoonright t', \text{ and } 0 = \Gamma_A(\langle x_m,y_m\rangle) = \Psi_i^{W_i\upharpoonright t'}(\langle x_m,y_m\rangle).
\]
Note that the current structure $\mathcal{A}_{s'}$ can be decomposed as follows:
\[
    \mathcal{A}_{s'} = \mathcal{A}^1_m + \mathcal{A}^{2,fin}_m = \mathcal{A}^{1,init}_m + \mathcal{C}_m + \mathcal{D}_m + \mathcal{A}^{2,fin}_m,
\]
where $\mathcal{A}^{2,fin}_m$ contains the elements added after the end of Step~(2).

(4)\ We define $\mathcal{B}_m := \mathcal{A}_{s'}$. We extend $\mathcal{A}_{s'}$ by adding \emph{precisely one} fresh element between the rightmost element of $\mathcal{A}^{1,init}_m$ and the leftmost element of $\mathcal{C}_m$. 

Inside the resulting structure $\mathcal{A}^3_m$, the number $x_m$ becomes the leftmost element of a copy of the cycle $\mathcal{D}_m$. This implies that at the moment, we have $\langle x_m,y_m\rangle \in \Gamma_A$. Similarly to Step~(2), we restrict ``$\langle x_m,y_m\rangle \in \Gamma_A$".

(5)\ We wait for a stage $s''>s'$ witnessing the following condition: for some $t''\leq s''$, we have $t'\leq t''$ and
\[
    W_{i}\upharpoonright t'' = \Phi_i^{\Gamma_A}\upharpoonright t'', \text{ and } 1 = \Gamma_A(\langle x_m,y_m\rangle) = \Psi_i^{W_i\upharpoonright t''}(\langle x_m,y_m\rangle).
\]

When the stage $s''$ is found, we go back to Step~(2) with $m+1$ (in place of $m$), while \emph{simultaneously waiting} at Step~(6) with $m$.

(6) We wait for a stage $s'''>s''$ witnessing the following condition: one can embed the elements of $\mathcal{A}_{s'''}$ into $(\omega,<,f) \upharpoonright (2 \cdot\mathrm{card}(\mathrm{dom}(\mathcal{A}_{s'''}))+1)$ in the following special ``semi-isomorphic" way. There exists a 1-1 function $\xi: \text{dom}(\mathcal{A}_{s'''}) \to [0,2 \cdot\mathrm{card}(\mathrm{dom}(\mathcal{A}_{s'''}))]$ such that:
\begin{itemize}
    \item $\xi$ respects the ordering, i.e., for \emph{all} elements $x,y \in\mathcal{A}_{s'''}$, $x <_{\mathcal{A}_{s'''}} y \Leftrightarrow \xi(x) < \xi(y)$;

    \item if $x$ (originally) was an element of some block $I$ inside $\mathcal{A}_{s'} = \mathcal{B}_m$, then the whole block $I$ should go into some copy of $I$ inside $(\omega,<,f)$; in other words,
    the restriction of $\xi$ to the domain of $\mathcal A_{s'}$ is an isomorphic embedding from $(\mathrm{dom}(\mathcal{A}_{s'}), <_{\mathcal{A}_{s'}},f_{\mathcal{A}_{s'}})$ into $(\omega,<,f)$.
\end{itemize}

Note the following: if such a ``semi-isomorphic" embedding $\xi$ exists, then one may assume that this $\xi$ satisfies an additional condition:
\begin{itemize}
    \item[($\ast$)] If $I$ is a block inside $\mathcal{A}^{1,init}_m$, then $I$ is mapped to its counterpart inside $(\omega,<,f)$ (or more formally, if $x\in I$ and $x$ is the $i$-th element from the left inside $\mathcal{A}_{s'''}$, then $\xi(x)$ equals $i$).
\end{itemize}
Indeed, since the structure $\mathcal{A}^3_m$ from Step~(4) is obtained by adding only one element just before $\mathcal{C}_m$, the embedding $\xi$ must move the contents of $\mathcal{C}_m$  to the right of the $(\omega,<,f)$-counterpart of $\mathcal{A}^{1,init}_m$. This allows us not to move $\mathcal{A}^{1,init}_m$, and just map it to its $(\omega,<,f)$-counterpart.

(7)\ Extend $\mathcal{A}_{s'''}$ to a finite structure by using the ``semi-isomorphic" embedding $\xi$ described above. More formally, we take the least unused numbers $y\not\in\mathrm{dom}(\mathcal{A}_{s'''})$ to extend $\mathcal{A}_{s'''}$ to a finite structure $\mathcal{A}^4_m$ such that there exist a number $N$ and an isomorphism $h\colon\mathcal{A}^4_m \cong (\omega,<,f)\upharpoonright N$ with the property $\xi\subseteq h$. Stop the strategy, including the actions for all $m'\neq m$.

\emph{Outcomes:} 

$w_m$: Waiting at Step~(3) forever for this $m$. Then either $W_i\neq \Phi_i^{\Gamma_A}$ or $\Gamma_A \neq \Psi_i^{W_i}$.

$w'_m$: Waiting at Step~(5) forever for this $m$. Then, again, $W_i\neq \Phi_i^{\Gamma_A}$ or $\Gamma_A \neq \Psi_i^{W_i}$

$s$: ``Stop", i.e., some $m$ reached Step~(7). Then we have:
\begin{itemize}
    \item $0 = \Gamma_A(\langle x_m,y_m\rangle) = \Gamma_{A,s'''+1}(\langle x_m,y_m\rangle) = \Psi_{i,s'}^{W_{i,s'}\upharpoonright t'}(\langle x_m,y_m\rangle)$.
    
    \item Let $u$ be the use for the computation $\Psi_{i,s'}^{W_{i,s'}\upharpoonright t'}(\langle x_m,y_m\rangle) = 0$ at Step~(3). Notice that $u\leq t'$. Since at Step~(5) we see $\Psi_{i,s''}^{W_{i,s''}\upharpoonright t''}(\langle x_m,y_m\rangle) = 1$, this implies that there exists an element $a\leq u$ such that $a\in W_{i,s''}\setminus W_{i,s'}$.
    
    \item Since $a\not\in W_{i,s'}$, at Step~(3) we have $0 = \Phi_{i,s'}^{\Gamma_{A,s'}}(a)$. Let $v$ be the use for the computation $\Phi_{i,s'}^{\Gamma_{A,s'}}(a) = 0$.
    
    \item The embedding $\xi$ from Step~(6) guarantees that $\Gamma_A \upharpoonright v = \Gamma_{A,s'}\upharpoonright v$. Hence, $W_i(a) = W_{i,s''}(a) = 1 \neq 0 = \Phi_i^{\Gamma_A}(a)$.
\end{itemize}
Therefore, the requirement $\mathcal{P}_i$ is satisfied.



$\infty$: Eventually waiting at Step~(6) for each $m\in\omega$. Then for each $m$, every $\mathcal{A}_{s'''}$ lacks an appropriate ``semi-isomorphic" embedding $\xi$. This means that each $\mathcal{B}_m$ can be isomorphically embedded only once into $(\omega,<,f)$. As discussed above, this contradicts Lemma~\ref{lemma:successor-recovery}.

The current outcome of a strategy is equal to:
\begin{itemize}
    \item $s$, if the strategy is already stopped.
    
    \item Otherwise, let $m$ be the current (maximal) value of our strategy parameter $m$. If for this $m$, we wait at Step~(3), then the outcome is $w_m$. If we wait at Step~(5), then the outcome is $w'_m$.
\end{itemize}

\textsc{Construction.} We use the following ordering of the finitary outcomes: $s < \dots < w'_2 < w_2 < w'_1< w_1 < w'_0 < w_0$. 
The tree of strategies includes only the finitary outcomes. More formally, we set $\Lambda = \{ s\} \cup \{ w_m, w'_m \,\colon m\in\omega\}$, and the tree $T$ is equal to $\Lambda^{<\omega}$. The $i$-th level of the tree contains strategies $\alpha\in T$ devoted to the requirement $\mathcal{P}_i$.

If $\sigma$ and $\tau$ are two finite strings from $T$, then by $\sigma\widehat{\ }\tau$ we denote the concatenation of $\sigma$ and $\tau$. As usual, we say that $\sigma$ is \emph{to the left} of $\tau$ if there exist some $\rho\in T$ and $o_1,o_2 \in \Lambda$ such that $o_1 < o_2$, $\sigma\supseteq \rho\widehat{\ }o_1$, and $\tau \supseteq \rho\widehat{\ }o_2$.

As usual, at a stage $s$ of the construction we visit the strategies $\alpha_0,\alpha_1,\dots,\alpha_s$, where $\alpha_0 = \emptyset$, and for each $i<s$ we have $\alpha_{i+1} = \alpha_i \widehat{\ }o$, where $o$ is the current outcome of the strategy $\alpha_i$. By $g_s$ we denote the current finite path, i.e., the sequence $(\alpha_0,\alpha_1,\dots,\alpha_s)$.

\textsc{Verification.} It is clear that the constructed $\mathcal{A}$ is a computable linear order on $\omega$: indeed, if $x<y$ inside $\mathcal{A}_s$ for some $s\in\omega$, then $x <_{\mathcal{A}_t} y$ for all $t\geq s$.

Let $g$ be the true path of the construction, i.e., the limit $\lim_s g_s$. More formally, for $k\in\omega$ and $\alpha\in T$, here we have
\[
    g(k) = \alpha\ \Leftrightarrow\ \exists t (\forall s\geq t) (g_s(k) = \alpha).
\]
Note that in general, $g$ could be a finite sequence. 

\begin{lemma}\label{lemma:inf-path}
    The path $g$ is infinite, and every requirement $\mathcal{P}_i$ is satisfied.
\end{lemma}
\begin{proof}
    Suppose that a strategy $\alpha$ belongs to the true path $g$, and its associated requirement is $\mathcal{P}_i$. Consider the following three cases.
    
    \emph{Case~1.} There is a number $m$ such that starting from some stage $s$, the outcome of $\alpha$ is always the same $o \in \{ w_m,w_m'\}$. Then it is clear that $\alpha\widehat{\ }o$ belongs to $g$. In addition, for the number $m$, the strategy is forever stuck either at Step~(3) or at Step~(5). This means that $W_i\neq \Phi_i^{\Gamma_A}$ or $\Gamma_A \neq \Psi_i^{W_i}$.
    
    \emph{Case~2.} At some point of time $\alpha$ has outcome $s$. Then $\alpha\widehat{\ }s$ lies on the path~$g$. In addition, $W_i \neq \Phi_i^{\Gamma_A}$ as discussed in the description of the outcome $s$.
    
    \emph{Case~3.} Otherwise, for each $m\in\omega$, $\alpha$ eventually goes through the outcomes $w_m$ and $w'_m$. Since we never reach Step~(7), this means that each finite structure $\mathcal{B}_m$ can be isomorphically embedded into $(\omega,<,f)$ only once: indeed, if some $\mathcal{B}_m$ could be embedded twice, then we could eventually use the second such embedding to recover the ``semi-isomorphic" map $\xi$ and to reach Step~(7) for this $m$. Then, as discussed in the beginning of the proof of the theorem, Lemma~\ref{lemma:successor-recovery} guarantees that Case~3 is impossible.
    
    We conclude that the path $g$ is infinite, and each $\mathcal{P}_i$ is satisfied.
\end{proof}

In order to finish the proof of Theorem~\ref{theorem:non-c.e.}, now it is sufficient to show that the order $\mathcal{A} = (\omega, <_A)$ is isomorphic to $(\omega,<)$. 

Recall that at Step~(6), we always choose a map $\xi$ satisfying Condition~($\ast$). Consider a $P_i$-strategy $\alpha$. Condition~($\ast$) implies that for every $m\in\omega$ and every $x\in \mathcal{A}^{1,init}_m$, $\alpha$ never adds new elements which are $<_A$-below $x$.

Suppose that an element $x$ is added to $\mathcal{A}$ by some strategy $\sigma$.

Consider an arbitrary strategy $\alpha$. If $\alpha$ is to the right of $\sigma$, then $\alpha$ never works after the starting stage of $\sigma$. If $\alpha \supset \sigma$ or $\alpha$ is to the left of $\sigma$, then $x$ always belongs to $\mathcal{A}^{1,init}_m$ of this particular $\alpha$. Hence, $\alpha$ never adds elements $<_{A}$-below $x$.

We deduce that new elements which are $<_A$-below $x$ could be added only by $\alpha\subseteq \sigma$. The proof of Lemma~\ref{lemma:inf-path} implies that each such $\alpha$ adds only finitely many elements to $\mathcal{A}$. Therefore, for an arbitrary element $x$, $\mathcal{A}$ contains only finitely many elements $<_A$-less than $x$. This implies that $(\omega,<_A)$ is isomorphic to $(\omega,<)$. This concludes the proof for the case when each $f$-block is a cycle.


Now we discuss the general case of the theorem. The proof essentially follows the outline provided above, but we have to address two important details of how to implement the strategy for $\mathcal{P}_i$.

(i)\ In Step~(2), one needs to choose two adjacent blocks $\tilde{\mathcal{C}}_m$ and $\tilde{\mathcal{D}}_m$. The question is how to \emph{algorithmically} choose them?

Here the computability of the function $cp_f(x)$ is important---this fact guarantees that for a given  $x\in\omega$, one can effectively recover the $f$-block $I$ containing $x$. This effective recovery allows us to computably find the needed blocks $\tilde{\mathcal{C}}_m$ and $\tilde{\mathcal{D}}_m$.

The recovery of the block $I$ can be arranged as follows. Without loss of generality, we may assume that $x$ is the leftmost element of $I$. Since we know the value $cp_f(x)$, we can find all elements from the preimage $f^{-1}(x)$. By using the function $cp_f$ several times, we eventually find the \emph{finite} set
\[
    P_f(x) = \bigcup_{j\in\omega} (f^{-1})^{(j)} (x).
\]
If $P_f(x)$ already forms an $f$-block, then we stop the algorithm. Otherwise, there is (the least) $y_0\not\in P_f(x)$ such that $x<y_0<\max P_f(x)$. We find the finite set $P_f(y_0)$. If the set $P_f(x)\cup P_f(y_0)$ forms an $f$-block, then we stop. Otherwise, again, we find the least $y_1\not\in P_f(x)\cup P_f(y_0)$ such that $y_0<y_1<\max(P_f(x)\cup P_f(y_0))$. We consider $P_f(x) \cup P_f(y_0) \cup P_f(y_1)$, etc. Since every $f$-block is finite, eventually we will find the desired $f$-block $I$.

(ii)\ Since the block function $f$ is not almost identity, $f$ satisfies at least one of the following two conditions:

(a) There are infinitely many elements $u$ such that $u$ is the leftmost element of its block and $f(u)>u$.
    
(b) There are infinitely many pairs $(u,v)$ such that $u$ and $v$ belong to the same block, $u$ is the leftmost element of the block, $u < v$, and $f(v) = u$.

Assume that $f$ satisfies~(b) (the case of~(a) could be treated in a similar way). Then in Step~(2) of the strategy, we can always choose $\tilde{\mathcal{C}}_m, \tilde{\mathcal{D}}_m$ with the following property: the block $\tilde{\mathcal{D}}_m$ contains a pair $(\tilde{u}, \tilde{v})$ satisfying the condition described in item~(b).

After building $\mathcal{A}^1_m$ (as described in Step~(2)), we will choose $x_m$ and $y_m$ as follows. The element $y_m$ is the rightmost element of $\mathcal{C}_m$, and the element $x_m$ is the copy (inside $\mathcal{D}_m$) of the element $\tilde{v} \in \tilde{\mathcal{D}}_m$. 

The intention behind this particular choice of $x_m,y_m$ is the following. At the end of Step~(2), we have $\langle x_m, y_m\rangle \not\in \Gamma_A$. \emph{In addition}, if we add \emph{precisely one} element at Step~(4), then we will immediately obtain that the condition $\langle x_m, y_m\rangle \in \Gamma_A$ becomes satisfied (since $\tilde{u} = f(\tilde{v})$ is the leftmost element of $\tilde{\mathcal{D}}_m$).

Taking the discussed details into account, one can arrange the proof of the general case in a straightforward manner.
Theorem~\ref{theorem:non-c.e.} is proved.
\end{proof}


\section{Conclusions}

The philosophical framework of Shapiro~\cite{shapiro_acceptable_1982} is based on the observation that typically, computations are performed on syntactic objects rather than on natural numbers themselves. This leads to the formal notion of a \emph{notation}. Among all notations, Shapiro isolated the class of acceptable notations---those notations have particularly nice computational properties.

Shapiro's notations became influential in the philosophy of computing. In this paper, we partly re-cast the framework of notations within the setting of computable structure theory. As a test case, we work with the structure $(\omega,<)$, i.e., natural numbers with the standard order (some of the early Shapiro's results are reproduced in Appendix~\ref{sec:A}). We hope that this will put computable structures on philosophers' radar and lead to cross-fertilization with formal philosophy. 

Our main technical contribution (Theorem~\ref{theorem:DgSp_Succ}) explores the connection between two natural properties which arise in computable structure theory. For a unary total computable function $f$, we look at the following properties: \ref{p1} when the degree spectrum of $f$ on $(\omega,<)$ contains precisely the c.e.\ degrees; and \ref{p2} the recoverability of successor from $f$. The second property can be framed as \emph{relativized} acceptability of notation. Theorem~\ref{theorem:non-c.e.} proves that for a large class $K$ of functions $f$, the two conditions given above are equivalent. Informally speaking, our proof uses the fact that the functions from the class $K$ have pretty tame combinatorial properties.

\setcounter{footnote}{0}

As a concluding remark, we note that in general, Theorem~\ref{theorem:non-c.e.} does not cover the class of all computable functions---even for the case of block functions.

\begin{theorem}[Appendix~\ref{sect:second-theorem}]\label{theorem:another-example}
    There exists a computable block function $f$ such that:
    \begin{enumerate*}
        \item the corresponding function $cp_f(x)$ is not computable;
        \item each $f$-block occurs infinitely often in $(\omega,<,f)$; 
        \item the degree spectrum of $f$ on $(\omega,<)$ contains all $\Delta_2$ degrees.
    \end{enumerate*}
\end{theorem}

Nevertheless, we conjecture that the equivalence of \ref{p1} and \ref{p2} could be established for the class encompassing all functions of interest:

\begin{conjecture} 
If the degree spectrum of a computable block function $f$ on $(\omega,<)$ is equal to all c.e.\ degrees, then the successor is recoverable from $f$ on $(\omega,<)$.
\end{conjecture}
\bibliographystyle{plain}
\bibliography{ref}

\appendix

\section{Shapiro's results in computable structure theory}\label{sec:A}

In this section, we consider notations for plain $\omega$. Let $L$ be a computable language (also called computable signature).

\begin{definition}[cf. Definition \ref{dgsp1}]\label{dgsp2}
For a countably infinite $L$-structure $\mathcal{A}$, its \emph{degree spectrum} is the set of all Turing degrees $\mathbf{d}$ such that there is an $L$-structure $\mathcal{B}$ with the following properties:
\begin{enumerate*}
	\item $\mathcal{B}$ is isomorphic to $\mathcal{A}$,
	\item the domain of $\mathcal{B}$ equals $\omega$,
	\item the Turing degree of the atomic diagram of $\mathcal{B}$ is equal to $\mathbf{d}$ (here a formula from the atomic diagram is identified with its G{\"o}del number).
\end{enumerate*}
\end{definition}

\paragraph{\textbf{\textsc{Theorem~4 of~\cite{knight_degrees_1986} (Knight).}}}
	\emph{A structure $\mathcal{A}$ is called \emph{automorphically trivial} if there is a finite set $X \subseteq \mathrm{dom}(\mathcal{A})$ such that every permutation of $\mathrm{dom}(\mathcal{A})$ that fixes $X$ is an automorphism of $\mathcal{A}$. If a countably infinite structure $\mathcal{A}$ is not automorphically trivial, then its degree spectrum is closed upwards in Turing degrees.}
	
\paragraph{\textbf{\textsc{T1 and C1 in \cite{shapiro_acceptable_1982}.}}} \emph{A function $f: \omega \to \omega$ is computable in every notation if and only if either $f$ is almost constant (i.e., there is $c$ such that $f(x) = c$ holds for all but finitely many $x$), or $f$ is almost identity (i.e., $f(x) = x$ for all but finitely many $x$). A relation $R \subseteq \omega$ is computable in every notation if and only if either $R$ is finite, or $R$ is cofinite.}
\begin{proof}
We only consider the nontrivial direction $(\Rightarrow)$. Suppose that $f$ is computable in every notation. Fix some standard notation $\sigma_0$---e.g., the one induced by decimal numerals. Our function $f$ must be computable in $\sigma_0$. Consider the language $L = \{ F\}$, where $F$ is a unary functional symbol. We define an $L$-stru\-cture $\mathcal{S}_f$ as follows: the domain of $\mathcal{S}_f$ equals $\omega$, and the functional symbol $F$ is interpreted as our function $f$. Since $f(x)$ is computable in $\sigma_0$, it is easy to show that the structure $\mathcal{S}_f$ is computable.

    Now, if $f(x)$ is neither almost constant nor almost identity, then one can show that the corresponding structure $\mathcal{S}_f$ is not automorphically trivial. Thus, by the theorem of Knight, one can obtain an isomorphic copy $\mathcal{A}$ of $\mathcal{S}_f$ such that the atomic diagram $D(\mathcal{A})$ of the structure $\mathcal{A}$ is Turing equivalent to, say, the Halting problem. Now we define a notation $\sigma: S \to \omega$: 
	\begin{enumerate*}
		\item $S$ equals (the decimal representation of) $\omega$, and
		\item $\sigma$ is an arbitrary isomorphism from $\mathcal{A}$ onto $\mathcal{S}_f$.
	\end{enumerate*} 
	Since $D(\mathcal{A})$ is not computable, one can show that the function $f(x)$ is not computable in $\sigma$.
\end{proof}

To establish Shapiro's criterion of acceptability, we need the notion of computable categoricity which goes back to the papers by Mal'tsev~\cite{maltsev_constructive_1961,maltsev_recursive_1962}.

\begin{definition}
	A computable $L$-structure $\mathcal{A}$ is \emph{computably categorical} if for every computable $L$-structure $\mathcal{B}$, which is isomorphic to $\mathcal{A}$, there is a computable isomorphism $g$ from $\mathcal{B}$ onto $\mathcal{A}$ (i.e., $g$ is an isomorphism which is also a partial computable function).
\end{definition}

Mal'tsev (essentially Theorem~4.1.2 in~\cite{maltsev_constructive_1961}) showed that every finitely generated, computable algebraic structure is computably categorical.

\paragraph{\textbf{\textsc{Acceptance criterion \cite{shapiro_acceptable_1982}.}}}
	\emph{For any notation $\sigma$, the following are equivalent:}
	\begin{itemize}
		\item[(a)] \emph{For any function $f$, $f$ is computable in $\sigma$ if and only if $f$ is computable in the standard notation (e.g., the stroke notation).}
		\item[(b)] \emph{The successor function is computable in $\sigma$.}
	\end{itemize}
\begin{proof}
We sketch the proof for the direction (b)$\Rightarrow$(a). Let $Succ$ denote the successor function. Since the structure $(\omega, Succ)$ is one-generated, by the result of Mal'tsev, $(\omega,Succ)$ is computably categorical. Now consider a notation $\sigma: S \to \omega$ such that the successor function  is computable in $\sigma$. Let $Succ^\sigma$ be the image of the successor under $\sigma^{-1}$. We define a computable $L$-structure $\mathcal{T}_{S,\sigma}$ as follows.
\begin{itemize}
	\item The domain of $\mathcal{T}_{S,\sigma}$ equals $\omega$.
	\item Fix a computable bijection $\psi$ from $\omega$ onto $S$. The unary functional symbol $F$ is interpreted as follows: for $k\in\omega$,
	\[
		F(k) = \psi^{-1}( \mathrm{Succ}^{\sigma}(\psi(k)) ).
	\]
\end{itemize}
It is not hard to show that the structure $\mathcal{T}_{S,\sigma}$ is a computable isomorphic copy of $(\omega,Succ)$. Then the computable categoricity of $(\omega,Succ)$ allows to choose (the unique) computable isomorphism $g$ from $\mathcal{T}_{S,\sigma}$ onto $(\omega,Succ)$.

After that, one can use the computability of $g$ to easily show that the notation $\sigma$ satisfies the conditions from the item~(a).
\end{proof}

\section{Proof of Theorem~\ref{theorem:another-example}} \label{sect:second-theorem}

\begin{proof}[Sketch]
    As usual, for $i\in\omega$, $\varphi_i$ denotes the unary partial computable function which has G{\"o}del number $i$.

    We build a computable block function $f$ satisfying the following series of requirements:
    \begin{itemize}
        \item[$\mathcal{P}_i$:] The function $cp_f$ is not equal to $\varphi_i$.
        
        \item[$\mathcal{R}$:] Each $f$-block occurs infinitely often inside $(\omega,<,f)$.
    \end{itemize}
    
    At a stage $s$, we define the finite function $f\upharpoonright N_s$ (where $N_s\in\omega$ and $N_s < N_{s+1}$) in such a way that $(\{0,1,\dots,N_s-1\}, <, f\upharpoonright N_s)$ consists of blocks.
    
    Beforehand, we put $f(0) = 0$, $f(1) = 2$, and $f(2)=1$.
    
    The $\mathcal{R}$-strategy is a global one. The $\mathcal{R}$-requirement is satisfied in a simple way, as follows. At the end of each construction stage $s$, let $B_s$ be the current set of all (isomorphism types of) blocks. Then (before starting the stage $s+1$) we use fresh numbers $x$ (i.e., the least numbers such that $f(x)$ is not defined yet) to extend $f$ in the following way: for each block $I$ from $B_s$, we add \emph{precisely two} adjacent copies of $I$.
    
    Note that in addition to the $\mathcal{R}$-requirement, this procedure will guarantee that in the final structure $(\omega,<,f)$, there will be infinitely many pairs $(x,x+1)$ such that $f(x) = x$ and $f(x+1)=x+1$. This preliminary observation will help us in the verification.
    
    \textsc{Strategy for $\mathcal{P}_i$.} Suppose that the strategy starts working at a stage $s_0+1$.
    
    (1)\ Choose a large fresh size $l_i$ such that the cycle of size $l_i$ cannot be isomorphically embedded into the current structure $(\{0,1,\dots,N_{s_0}-1\}, <, f\upharpoonright N_{s_0})$. (See the proof of Theorem~\ref{theorem:non-c.e.} for the definition of a cycle).
    
    Extend $f$ by adding a copy of the cycle of size $l_i$. Let $w_i$ be the leftmost element of this copy.
    
    (2)\ Wait for a stage $s'>s_0+1$ such that $\varphi_{i,s'}(w_i)\!\downarrow\ =1$.
    
    (3)\ Extend $f$ by taking the least number $x$ such that $f(x)$ is still undefined, and setting $f(x) := w_i$.
    
    This concludes the description of the strategy. It is clear that it satisfies the requirement $\mathcal{R}_i$. Indeed, assume that the function $\varphi_i$ is total. If $\varphi_i(w_i) \neq 1$, then the element $w_i$ is a part of a cycle, and $cp_f(w_i) = 1 \neq \varphi_i(w_i)$. If $\varphi_i(w_i) = 1$, then we have $cp_f(w_i) = 2$.
    
    \textsc{Construction.} The construction is arranged as a standard finite injury argument. The requirements are ordered: $\mathcal{P}_0 < \mathcal{P}_1 < \mathcal{P}_2 < \dots$. When a higher priority strategy $\mathcal{P}_i$ acts (i.e., it extends $f$ in its Step~(3)), it initializes all lower priority strategies $\mathcal{P}_j$, $j>i$.
    
    \textsc{Verification.} The non-computability of the function $cp_f(x)$ can be proved by a standard argument for finite injury constructions.
    
    A more hard part is how to show that every $\Delta_2$ degree belongs to the degree spectrum of the constructed $f$. This can be achieved by arranging an argument similar to the argument of Case~(a) of Theorem~14 in~\cite{bazhenov_intrinsic_2022}. Roughly speaking, given an arbitrary $\Delta_2$ set $X$, we should encode it via a computable copy $\mathcal{A} = (\omega,<_A)$ as follows. For each $e\in\omega$, the numbers $4e$ and $4e+2$ will be $<_{A}$-adjacent, and more importantly: 
    \begin{itemize}
        \item if $e\in X$, then $f_{\mathcal{A}}(4e) = 4e+2$ and $f_{\mathcal{A}}(4e+2) = 4e$;
        
        \item if $e\not\in X$, then $f_{\mathcal{A}}(4e) = 4e$ and $f_{\mathcal{A}}(4e+2) = 4e+2$.
    \end{itemize}
    Our preliminary observation helps us to arrange this encoding. This concludes the proof of Theorem~\ref{theorem:another-example}.

\end{proof}

\end{document}